\documentclass[reqno,psamsfonts,12pt]{gen-j-l}
\usepackage{amssymb,amsmath,amsbsy,amsthm,esint,setspace}
\usepackage{hyperref}
\usepackage{amsrefs}
\usepackage{enumitem}

\newtheorem{theorem}{Theorem}[section]
\newtheorem{lemma}[theorem]{Lemma}
\newtheorem{proposition}[theorem]{Proposition}
\newtheorem{corollary}[theorem]{Corollary}

\newtheorem{example}{Example}[section]

\newtheorem*{remark}{Remark}

\begin{document}

\title[Virial estimates for hard spheres]{Virial estimates
for hard spheres}

\author{Ryan Denlinger}

\begin{abstract}
We review a virial-type estimate which bounds the
strength of interaction for a gas of $N$ hard spheres (billiard balls)
dispersing into Euclidean space $\mathbb{R}^d$.
This type of estimate has been known for
decades in the context of (semi-)dispersing
billiards, and is essentially trivial in that context.
Our goal, however, is to write virial estimates in
a way which may lend insight into the problem of rigorously
deriving Boltzmann's
equation (cf. Lanford's theorem). Using virial estimates, we provide a short
proof of lower bounds (sharp up to
powers of logarithms) on the convergence rate of the first marginal in
Lanford's theorem. Such lower bounds will often, but not always,
 follow trivially from
energy conservation; the proof we present holds assuming only that
the limiting dynamics is regular enough and does not reduce to
 free transport.
\end{abstract}

\maketitle

\section{Introduction}
\label{intro}

The problem of interest to us is that of deriving various
nonlinear partial differential equations (PDEs) starting from the
Newtonian gas of $N$ hard spheres. Depending on the chosen scalings, the
relevant PDE could be the Navier-Stokes-Fourier equations, Boltzmann's
equation, etc. (though fully nonlinear Navier-Stokes-Fourier is
far out of reach by current methods). 
 The existence and uniqueness of solutions to nonlinear PDEs
is generally an open problem, except in the presence of very special
conservation principles or perturbative assumptions. Even when
solutions are known, the analysis tends to be quite complicated,
depending on the strength of available \emph{a priori} estimates for
a hypothetical solution. For this reason, we are naturally led to
the problem of deriving analogous \emph{a priori} bounds on the
particle model.

\subsection{Hydrodynamic Limits of Interacting Particle Systems.}

There is not one unique way to approach the derivation of hydrodynamic
equations starting from mechanical laws.
 One possible strategy to attack this problem would be
to set up a hydrodynamic scaling at the particle level and let
$N\rightarrow \infty$ (with the hope that local Gibbs states will 
possess some ergodicity). This is sometimes a useful approach in
the presence of \emph{stochasticity} (e.g. see \cite{OVY1993})
but has not been particularly 
fruitful in the deterministic case due to the limited understanding
of dynamical systems in many dimensions. (Note however
that some one-dimensional
 models are tractable, e.g. identical hard
rods. \cite{BDS1983}) A second possible strategy for hydrodynamic limits
is to look for
an intermediate (kinetic) description, retaining some of the microscopic
information but not all of it. Kinetic descriptions operate on
much smaller timescales than hydrodynamic descriptions because
time averaging always washes some microscopic information away.
Therefore, in order to pass from a kinetic description to a hydrodynamic
description \emph{at the particle level},
we need \emph{quantitative bounds} on \emph{long time
intervals} (compactness is not enough!).

Despite striking advances in the passage from Boltzmann's equation
to hydrodynamic equations in various low-density regimes
(see \cite{LSR2009} for an overview), the derivation of Boltzmann's
equation from Newton's laws is still in its infancy.
A classical theorem due to O. E. Lanford establishes the validity
of Boltzmann's equation for a hard sphere gas, but only up to
a fraction of the mean free time for a particle of gas.
\cite{L1975,GSRT2014,PSS2014}
Obviously Lanford's theorem is completely unsatisfactory because
we need \emph{many} collisions even to progress past $t=0$ in 
a hydrodynamic description. R. Illner and M. Pulvirenti 
were able to obtain convergence globally in time, but only when
the gas is so diffuse that particles mostly do not collide at all.
\cite{IP1986,IP1989}
H. van Beijeren, O. E. Lanford, J. L. Lebowitz and H. Spohn were
able to derive the (non-conservative) \emph{linear} Boltzmann
equation for a tagged particle in an equilibrium background, as well
as the linearized Boltzmann equation which is formally associated to the response of the
background itself.
\cite{vBLLS1980,LS1982} 

Much more recently, T. Bodineau, I. Gallagher, and L. Saint-Raymond 
were able to quantify the convergence from \cite{vBLLS1980} on
time scales $T_N$ diverging like
a power of $\log \log N$, thereby deriving Brownian motion in a 
suitable hydrodynamic scaling. \cite{BGSR2015} (Note that the $\log \log N$
timescale is still troublesome from a physical point of view but
it is hard to avoid mathematically using the series-based
methods of \cite{L1975,vBLLS1980,BGSR2015}.)
 In a follow-up work, the same authors considered
a symmetrized perturbation
of size $1/N$ 
(this simulates the response of an equilibrium
background to the influence of a tagged particle). \cite{BGSR2015III}
This leads to a rigorous derivation of the \emph{linearized} Boltzmann
equation on long timescales, and subsequently a derivation of
 linear hydrodynamic models (in two dimensions only).

\subsection{Monotonicity, convexity, Morawetz, Bony.}

One of the classical problems for billiard systems (such as the
hard sphere gas) is to estimate the number of collisions in
a finite time interval. It is known that this number is finite
for $N$ hard spheres in $\mathbb{R}^d$ 
(see \cite{Va1979,I1989} for two different proofs) but 
\emph{a priori} it might depend on the initial condition.
Actually it turns out that the number of collisions is
bounded \emph{uniformly} with respect to initial conditions,
but might grow like $N^N$ or worse. \cite{BFK1998}
Note that \emph{even if we ignore high-order correlations}
and simply consider 
 clusters of $\mathcal{O} (\log N)$
particles, the function $(\log N)^{\log N}$ \emph{still} grows
faster than any power of $N$.
On the other hand, with respect to the Lanford theorem,
one does not \emph{really} care about the total number of collisions.
We care about \emph{estimates} in \emph{good function spaces};
we do not need to count all collisions the same way.

Workers in billiards theory have known for decades that
\emph{some} collisions can be estimated very efficiently
by constructing monotonic or convex functions of phase space
coordinates. (Monotonicity or convexity is here measured
along a fixed trajectory.) This idea
 was stated explicitly in \cite{BFK1998} and
was used implicitly in both \cite{Va1979} and \cite{I1989}.
It turns out that these (monotonic or convex)
 functions are of the same type as the
functions appearing in proofs of virial and Morawetz type
estimates for dispersive PDEs and Vlasov-type equations.
(In fact T. Tao points out the connection explicitly in
his book \cite{TTao}; his \S 1.5 Example 1.34 may be viewed as
a caricature of our Corollary \ref{cor:s3-IP-3}, whereas our
Corollary \ref{cor:s3-IP-3} is all but written
already by R. Illner in \cite{I1989}.)

There is no known analogue of virial or Morawetz identities
for the Boltzmann equation in general (without assuming
extra estimates above the energy level). The closest known results
are set in one space dimension; technically, the physical setting
is $\mathbb{R}^3$ with spatial variation along just one axis.
In that case, for certain collision kernels, it is possible to
write down an integral which effectively tracks the accumulation
of collisions as the solution $f(t)$ interacts with itself.
One can prove monotonicity in time, as well as uniform boundedness
in large time, using conservation laws. 
\cite{BCP2006,Ce2005,Ce1995,Ce1992}
This technique is
known as \emph{Bony's functional} or \emph{Glimm's functional},
by analogy with similar techniques for hyperbolic conservation
laws in one space dimension. The point of this technique is that,
in one dimension, two disturbances will perhaps pass through each
other a few times and interact, but each time some part of
the \emph{potential for interaction}
is expended and cannot be used again. This potential for interaction
can only be
measured directly
due to the one-dimensional geometry. In higher dimensions,  one would
have to consider potentials along many possible trajectories of
the system and this is just too difficult to quantify (compare the
difficulty of tracking shocks in higher dimensions).

\subsection{Main results and organization of this paper.}

The main focus of this work is Proposition \ref{prop:s4-interaction}, which
is a virial-type spacetime estimate for hard spheres. As noted
above, virial-type estimates
are essentially classical in the billiards literature, and they play a
prominent role in the derivation of Boltzmann's equation. 
(See \cite{IP1986,IP1989}, particularly the first lemma of \cite{IP1989},
which is Lemma \ref{lemma:s3-IP-1} in the present manuscript.) 
The main difference with Proposition \ref{prop:s4-interaction} is that the
classical virial bound is re-formulated to
 control a quantity closely associated with the hard sphere BBGKY
hierarchy, for a wide class of initial data.
Unfortunately, while these estimates are quite general, they do not
lead directly to coercive estimates at the limit for any nontrivial
scaling of which we are aware. We will show, however, that virial
estimates can be used to place \emph{lower} bounds on the convergence
rate in Lanford's theorem. Such lower bounds may, but do not always,
follow trivially from energy conservation; our result holds under
essentially minimal assumptions on the initial data. The types
of data which are newly covered by our
 result have a product structure
at the initial time, $f_0 (x,v) = \rho (x) m ( v )$, or are convex
combinations $f_0 (x,v) = \int d\alpha \rho_\alpha (x)
m_\alpha (v)$ where each $m_\alpha$ has the same conserved moments.
 (See Example \ref{ex:ex1}.)

Section \ref{sec:2} introduces the basic notation of this work, which
mostly follows the presentation of \cite{GSRT2014}.
Section \ref{sec:3} gives an elementary derivation of an identity 
due to
Illner \cite{I1989}. In Section \ref{sec:4}, we apply
this
identity in a heuristic manner to derive the virial-type spacetime estimate;
a rigorous proof may be found in \cite{De2016T}.
Section \ref{sec:fact} gives a very concise overview of Lanford's
theorem, \cite{L1975,GSRT2014}, and some recent
developments. Finally in Section \ref{sec:L} we use virial identities
to prove a bound from below on the convergence rate in Lanford's
theorem (including cases where such lower bounds would not follow
directly from energy conservation).

\section{Notation}
\label{sec:2}

Consider $N$ non-overlapping
hard spheres centered at positions $x_i \in \mathbb{R}^d$ with
velocities $v_i \in \mathbb{R}^d$ for $i=1,2,\dots,N$. The spheres are
considered to have identical mass and radius, and are in all other ways
physically indistinguishable. For convenience, we will assume without
loss that all particles have unit diameter. The collection of all
positions is a tuple $X_N$,
\begin{equation*}
X_N = \left( x_1,x_2,\dots,x_N\right)\in\mathbb{R}^{dN}
\end{equation*}
\begin{equation*}
V_N = \left( v_1,v_2,\dots,v_N\right)\in\mathbb{R}^{dN}
\end{equation*}
The classical phase-space coordinates of $i$th particle are
given by $z_i = (x_i,v_i)$, and the phase-space coordinates of the
whole gas are denoted
\begin{equation*}
Z_N = \left( z_1,z_2,\dots,z_N\right)\in
\mathbb{R}^{2dN}
\end{equation*}
We may also write $Z_N = \left( X_N,V_N\right)$. The following function
will play a central role in our analysis: for $t\in\mathbb{R}$ and
$Z_N \in \mathbb{R}^{2dN}$, we define
\begin{equation}
\label{eq:s2-r-N-def}
r_N \left( t,Z_N \right) =
\sum_{i=1}^N \left( x_i \cdot v_i - |v_i|^2 t\right)
\end{equation}
Following \cite{GSRT2014}, we may introduce the $N$-particle phase space $\mathcal{D}_N$,
which is defined by
\begin{equation}
\label{eq:s2-D-N-def}
\mathcal{D}_N = \left\{ \left. Z_N \in \mathbb{R}^{2dN} \right|
\; \forall 1\leq i < j \leq N,\;
|x_i - x_j| > 1  \right\}
\end{equation}
The choice of $\mathcal{D}_N$ is motivated by requirement that the
spheres be mutually disjoint at all times.
The closure of $\mathcal{D}_N$ in $\mathbb{R}^{2dN}$ in the standard
topology is denoted $\overline{\mathcal{D}}_N$, and we will also
write $\partial \mathcal{D}_N = \overline{\mathcal{D}}_N \backslash
\mathcal{D}_N$. We will use the notation $\textnormal{ a.e. }
 Z_N \in \mathcal{D}_N$
to refer to a typical point for the Lebesgue measure on
$\mathcal{D}_N$. The notation $\textnormal{ a.e. }
 Z_N \in \partial \mathcal{D}_N$
will refer to a typical point for the induced surface measure
 arising from the natural embedding
$\partial \mathcal{D}_N \subset \mathbb{R}^{2dN}$.

Formally speaking, we wish to solve Newton's laws with a hard core
interaction. This means if $Z_N(t_0) = (X_N(t_0),V_N(t_0)) \in \mathcal{D}_N$
then
\begin{equation*}
\begin{aligned}
\left. \frac{d}{dt} X_N (t) \right|_{t=t_0} &  = V_N (t_0) \\
\left. \frac{d}{dt} V_N (t) \right|_{t=t_0} & = 0
\end{aligned}
\end{equation*}
Hence the particles move freely between collisions. At each
collision (that is, $Z_N (t_0) \in \partial \mathcal{D}_N$), the
particles are required to interact elastically, thereby conserving
 momentum, energy, and angular momentum. The set of possible
interactions for two-body elastic collisions is easy to parametrize
explicitly. Suppose that there exists $i<j$ such that
 $x_j (t_0) = x_i (t_0) + \omega$ for some
$\omega \in \mathbb{S}^{d-1}$; and, further suppose that
$\left|x_{j^\prime} (t_0) -  x_{i^\prime}(t_0)\right| > 1$ for any
$i^\prime < j^\prime$ such that $(i^\prime,j^\prime) \neq (i,j)$.
Let us denote
\begin{equation*}
\begin{aligned}
\lim_{t\rightarrow t_0^-} V_N (t) & = \left( v_1,\dots,v_i,\dots,
v_j,\dots,v_N\right) \\
\lim_{t\rightarrow t_0^+} V_N (t) & = \left( v_1,\dots,v_i^*,\dots,
v_j^*,\dots,v_N \right)
\end{aligned}
\end{equation*}
Then we have
\begin{equation*}
\begin{aligned}
v_i^* & = v_i + \omega \omega \cdot \left( v_j - v_i \right) \\
v_j^* & = v_j - \omega \omega \cdot \left( v_j - v_i \right)
\end{aligned}
\end{equation*}

Similarly for $\textnormal{ a.e. } Z_N \in \partial \mathcal{D}_N$
we will use the notation $Z_N^*$ to refer to the image of the point
$Z_N$ through the collision transformation. The map 
$Z_N \mapsto Z_N^*$ is a measurable involution.

In the above ``definition,'' we have neglected to specify uniquely
what happens when more than two particles collide at the same time.
Multiple particle interactions occur with zero probability, though
this statement requires justification which we will not discuss.
(See \cite{Al1975} or \cite{GSRT2014}.)
The hard sphere flow at time $t$ defines a measurable map
\begin{equation*}
\psi_N^t : \mathcal{D}_N \rightarrow \mathcal{D}_N
\end{equation*}
For each $t\in\mathbb{R}$, the map $\psi_N^t$ preserves the  Lebesgue
measure on $\mathcal{D}_N \subset \mathbb{R}^{2dN}$.
Complete proofs of the existence of the hard sphere flow
$\psi_N^t$ may be found in the literature. \cite{Al1975,GSRT2014} 

Following Boltzmann's great insight, we realize that it is not very
interesting to discuss any \emph{particular} trajectory
$\left\{ \psi_N^t Z_N \right\}_{t\geq 0}$, because it is physically
infeasible (or impossible) to measure the positions and velocities of all
the particles at a given instant. Therefore, the initial
value problem for Newton's laws is \emph{not} the correct problem for
us to solve. The correct approach is to place a probability density
$f_N (0,Z_N)$ on the set of possible initial states $Z_N \in \mathcal{D}_N$.
The function $f_N (0,Z_N)$ represents our uncertainty
about the actual state of the system. Since we have no physical
means to distinguish between two particles in our model, the function
$f_N (0,Z_N)$ must be \emph{symmetric} with respect to interchange of
particle indices.

We will denote by $\mathcal{S}_N$ the symmetric group on $N$ letters.
Any permutation $\sigma \in \mathcal{S}_N$ acts on the phase-space
coordinates $Z_N = (z_1,z_2,\dots,z_N) \in \mathcal{D}_N$ as follows:
\begin{equation*}
\sigma Z_N = \left( z_{\sigma(1)},z_{\sigma(2)},\dots,
z_{\sigma(N)}\right) \in \mathcal{D}_N
\end{equation*}
Similarly, if $f_N (Z_N)$ is any function on $\mathcal{D}_N$, then
$\sigma$ acts on $f_N$ by composition: $\sigma f_N =
f_N \circ \sigma$. Let $\mathcal{P}\left(\mathcal{D}_N\right)$ denote the
set of probability measures on $\mathcal{D}_N$, and furthermore let
$\mathcal{P}_{\textnormal{a.c.}}\left(\mathcal{D}_N\right)$ denote the
set of probability measures which are absolutely continuous with respect
to the Lebesgue measure on $\mathcal{D}_N$. Any element of 
$\mathcal{P}_{\textnormal{a.c.}}\left(\mathcal{D}_N\right)$ may be
represented uniquely ($\textnormal{ a.e. } Z_N \in \mathcal{D}_N$)
by a non-negative function $f_N (Z_N)$ such that
$\int_{\mathcal{D}_N} f_N (Z_N) dZ_N = 1$. Finally let
$\mathcal{P}_{\textnormal{a.c.}}^{\textnormal{sym}} \left(
\mathcal{D}_N\right)$ be
the set of absolutely continuous measures on $\mathcal{D}_N$ such that
the associated function $f_N$ is invariant under the action of 
$\mathcal{S}_N$. Henceforth, when we write $f_N$, we will always mean
an element of $\mathcal{P}_{\textnormal{a.c.}}^{\textnormal{sym}}
\left(\mathcal{D}_N\right)$.

Let $f_N (0)$ be any element of 
$\mathcal{P}_{\textnormal{a.c.}}^{\textnormal{sym}}
\left( \mathcal{D}_N\right)$, which we regard as the initial state of
the $N$ particle gas. For any $t\in\mathbb{R}$ we will let $f_N (t)$
be the pushforward of $f_N (0)$ under the hard sphere flow 
$\psi_N^t$; then, $f_N (t)$ is likewise an element of
$\mathcal{P}_{\textnormal{a.c.}}^{\textnormal{sym}}
\left(\mathcal{D}_N\right)$. Since $\psi_N^t$ preserves the
Lebesgue measure on $\mathcal{D}_N$, we may write the following
expression for $f_N (t)$:
\begin{equation}
\label{eq:f-N-t}
f_N (t,Z_N) = f_N \left(0,\psi_N^{-t} Z_N\right)
\end{equation}
The functions $f_N (0)$ and $f_N (t)$ may be extended by zero so as
to be defined on $\mathbb{R}^{2dN}$.

For any
$1\leq s \leq N$, we define the marginal $f_N^{(s)} (t)$
by partial integration:
\begin{equation}
\label{eq:s2-f-N-s}
f_N^{(s)} (t,Z_s) =
\int_{\mathbb{R}^{2d(N-s)}} f_N (t,Z_N) dz_{s+1} \dots dz_N
\end{equation}
The evolution of the marginals $f_N^{(s)}(t)$ may be described
explicitly via the so-called BBGKY hierarchy 
(Bogoliubov-Born-Green-Kirkwood-Yvon) \cite{GSRT2014}, though we
will \emph{not} be making any use of the BBGKY hierarchy except in
Section \ref{sec:L}.
The marginals $f_N^{(s)} (t)$ are non-negative symmetric functions
on $\mathbb{R}^{2ds}$ with unit mass.

The main result we will show, Proposition \ref{prop:s4-interaction},
  will control the trace of the marginals
$f_N^{(s)} (t,Z_s)$
along a certain hypersurface in $\mathbb{R}\times \mathbb{R}^{2ds}$,
with polynomial dependence on $N$ for large values of $N$. This is slightly problematic
because the trace of an $L^1$ function is simply not defined; moreover,
even if the data $f_N (0)$ is \emph{smooth}, the function
$f_N (t)$ typically develops singularities. Nevertheless, due to
technical arguments which we will not discuss, it is possible to show
that if $f_N (0)$ is smooth and compactly supported in
$\mathcal{D}_N$ then the required traces of $f_N^{(s)}(t)$
do, in fact, exist (at least for almost every $t\in \mathbb{R}$).
 See \cite{CGP1997,GSRT2014,Si2014,PS2015II} for more
information on regularity issues for hard spheres. Our estimates do not
depend on the choice of regularization, except insofar as the regularized
marginals must be a sequence of symmetric non-negative functions
which are indeed marginals in the sense of (\ref{eq:s2-f-N-s}).
Therefore, similar to the proof of the classical trace theorem in partial
differential equations ($W^{1,p} (U) \subseteq 
L^p (\partial U)$ for sufficiently smooth
 bounded regions $U \subset \mathbb{R}^k$),
the traces are actually meaningful \emph{for solutions of Liouville's equation}
even if the initial data is only $L^1$. We will not discuss further the issues of
regularity.

\section{A monotonicity formula}
\label{sec:3}

Let us fix an initial point $Z_N \in\mathcal{D}_N$ in the microscopic
phase-space, and consider the trajectory $\left\{
\psi_N^t Z_N \right\}_{t\in\mathbb{R}}$. Our analysis begins with a
simple observation: with $r_N (t,Z_N)$ as in (\ref{eq:s2-r-N-def}),
if we define
\begin{equation}
r_{Z_N} (t) = r_N \left( t,\psi_N^t Z_N\right)
\end{equation}
then for any $t_0$ such that $\psi_N^{t_0} Z_N \in \mathcal{D}_N$ we have
\begin{equation}
\left. \frac{d}{dt} r_{Z_N} (t) \right|_{t=t_0} = 0
\end{equation}
Indeed, we see that if $\dot{x}_i = v_i$ and $\dot{v}_i = 0$ then
\begin{equation*}
\frac{d}{dt} \left( x_i \cdot v_i - |v_i|^2 t\right) = 0
\end{equation*}
Therefore, the difference $r_{Z_N} (t) - r_{Z_N} (0)$ is simply equal
to a sum along \emph{collisions} of incremental jumps in 
$r_{Z_N}(t)$. It will turn out that all of these jumps have the
\emph{same sign}, and we can compute the jumps explicitly in terms
of collision parameters.

Let us compute the jump in $r_{Z_N}(t)$ across a collision taking place
at time $t_0 \in \mathbb{R}$. We may assume
that the interacting particles are simply those labelled $i=1,2$, since
collisions are binary and particles are indistinguishable.
The position coordinates are continuous in time, so
we write them $x_1,x_2$, with $x_2 = x_1 + \omega$ for some
$\omega \in \mathbb{S}^{d-1}$. The pre-collisional velocities will
be denoted $v_1\equiv v_1 (t_0^-),v_2 \equiv v_2(t_0^-)$ and the
post-collisional velocities will be denoted 
$v_1^*\equiv v_1 (t_0^+),v_2^* \equiv v_2 (t_0^+)$. We have
\begin{equation*}
\begin{aligned}
& r_{Z_N} (t_0^+) - r_{Z_N} (t_0^-) 
 = \left\{ \left( x_1 \cdot v_1^* - |v_1^*|^2 t_0\right) +
\left( x_2 \cdot v_2^* - |v_2^*|^2 t_0\right)\right\} + \\
& \qquad \qquad \qquad \qquad\qquad \quad
 + \left\{ - \left( x_1 \cdot v_1 - |v_1|^2 t_0\right) - 
\left( x_2 \cdot v_2 - |v_2|^2 t_0 \right)\right\}
\end{aligned}
\end{equation*}
Due to energy conservation,
\begin{equation*}
|v_1^*|^2 + |v_2^*|^2 = |v_1|^2 + |v_2|^2
\end{equation*}
so we may eliminate the explicit dependence on $t_0$.
\begin{equation*}
r_{Z_N} (t_0^+) - r_{Z_N} (t_0^-) =
x_1 \cdot v_1^* + x_2 \cdot v_2^* - x_1 \cdot v_1 - x_2 \cdot v_2
\end{equation*}
Since $x_2 = x_1 + \omega$, this gives us
\begin{equation*}
r_{Z_N} (t_0^+) - r_{Z_N} (t_0^-) =
x_1 \cdot \left( v_1^* + v_2^* - v_1 - v_2\right) +
\omega \cdot \left( v_2^* - v_2\right)
\end{equation*}
Due to momentum conservation,
\begin{equation*}
v_1^* + v_2^* = v_1 + v_2
\end{equation*}
so we may eliminate the explicit dependence on the position coordinates.
Hence
\begin{equation*}
r_{Z_N} (t_0^+) - r_{Z_N} (t_0^-) = \omega \cdot
\left( v_2^* - v_2\right)
\end{equation*}
This is the same as
\begin{equation*}
r_{Z_N} (t_0^+) - r_{Z_N} (t_0^-) = - \omega \cdot
\left( v_2 - v_1\right)
\end{equation*}
by the collisional change of variables from Section \ref{sec:2}.
But $v_1,v_2$ are the velocities of the two particles \emph{coming
into} a collision, so we must have
\begin{equation*}
\left( x_2 - x_1 \right) \cdot \left( v_2 - v_1\right) \leq 0
\end{equation*}
and therefore $\omega \cdot \left( v_2-v_1\right) \leq 0$. Hence,
\begin{equation*}
r_{Z_N} (t_0^+) - r_{Z_N} (t_0^-) = \left| \omega \cdot
\left( v_2 - v_1\right)\right|\geq 0
\end{equation*}

Adding up all collisions along the trajectory we obtain the following
identity, which was observed by Illner: \cite{I1989}

\begin{proposition}
\label{prop:s3-Illner}
(Illner) 
For $\textnormal{ a.e. } Z_N \in \mathcal{D}_N$ and 
$\textnormal{ a.e. }t\geq 0$ there holds
\begin{equation}
\label{eq:s3-Illner}
r_{Z_N} (t) - r_{Z_N} (0) = \sum_k \left|\omega_k \cdot
\left( v_{j_k} (t_k^-) - v_{i_k} (t_k^-)\right)\right|
\end{equation}
where the sum $\sum_k$ is over all collisions along the trajectory
$\left\{ \psi_N^\tau Z_N \right\}_\tau$ for $0 \leq \tau \leq t$.
\end{proposition}

We will require an auxiliary lemma due to Illner and Pulvirenti
which follows easily from Proposition \ref{prop:s3-Illner}. \cite{IP1986,IP1989}
In order to state the lemma, we introduce a new function on $\mathcal{D}_N$,
\begin{equation}
\label{eq:s3-I-N}
I_N (Z_N) = \sum_{i=1}^N |x_i|^2
\end{equation}
The proof is a computation, which we include for completeness.
\begin{lemma}
\label{lemma:s3-IP-1}
For $\textnormal{ a.e. } Z_N =(X_N,V_N)
\in\mathcal{D}_N$ and all $t\in\mathbb{R}$, we have
\begin{equation}
\label{eq:s3-IP-1}
I_N \left( \psi_N^t Z_N\right) \geq I_N \left(
(X_N + V_N t,V_N)\right) 
\end{equation}
\end{lemma}
\begin{proof}
By time-reversibility we may assume $t\geq 0$. The function
$I_N (\psi_N^t Z_N)$ is globally continuous in $t$ for
$\textnormal{ a.e. }Z_N \in \mathcal{D}_N$. With this in mind, it
suffices to point out that the desired inequality is true for $t=0$, and
between collisions (using energy conservation) we have
\begin{equation*}
\frac{d}{dt} \left\{ I_N \left( \psi_N^t Z_N \right) -
I_N \left( (X_N + V_N t,V_N)\right) \right\} =
2 \left\{ r_{Z_N}(t) - r_{Z_N} (0)\right\}
\end{equation*}
We conclude by Proposition \ref{prop:s3-Illner}.
\end{proof}

\begin{remark}
Lemma \ref{lemma:s3-IP-1} was the key estimate which Illner and Pulvirenti
\cite{IP1986,IP1989}
relied upon to rigorously derive Boltzmann's equation, globally in
time, for a rarefied gas in $\mathbb{R}^d$. The theorem of
Illner and Pulvirenti is analogous to ``small data'' results for
nonlinear PDE, and therefore does not resolve the problem of
deriving Boltzmann's equation near global Maxwellians.
Note however that Proposition \ref{prop:s3-Illner} is
\emph{more} general than Lemma \ref{lemma:s3-IP-1} and generally contains
more detailed information about the dynamics, including possible
cancellations. Another important point is that Illner and Pulvirenti
actually applied Lemma \ref{lemma:s3-IP-1} to isolated
\emph{clusters} of particles, controlling separately the interactions
between clusters. In the same way Proposition \ref{prop:s3-Illner}
is applicable to isolated clusters of particles, just as it is
applicable to the gas as a whole.
\end{remark}

The next lemma is technical, and again follows from 
Proposition \ref{prop:s3-Illner}.
\begin{lemma}
\label{lemma:s3-IP-2}
For $\textnormal{ a.e. }Z_N = (X_N,V_N) \in \mathcal{D}_N$,
all $t\geq 0$, and
all $\lambda > 0$, there holds
\begin{equation}
\label{eq:s3-IP-2}
\left| r_{Z_N} (t) \right| \leq \frac{1}{2} \lambda^{-1} \sum_{i=1}^N
\left( \lambda^2 |x_i|^2 + |v_i|^2 \right)
\end{equation}
\end{lemma}
\begin{proof}
Recall that $r_{Z_N} (t) \equiv r_N \left( t,\psi_N^t Z_N\right)$.
On the other hand, by (\ref{eq:s2-r-N-def}),
\begin{equation*}
\begin{aligned}
\left| r_N (t,Z_N) \right| & \leq \sum_{i=1}^N \left|
x_i \cdot v_i - |v_i|^2 t\right| \\
& = \sum_{i=1}^N \left| \left( x_i - v_i t\right)\cdot v_i \right| \\
& \leq \frac{1}{2} \sum_{i=1}^N \left( \lambda |x_i-v_i t|^2 + 
\lambda^{-1} |v_i|^2\right) \\
& = \frac{1}{2} \lambda I_N \left( (X_N-V_N t,V_N)\right) +
\frac{1}{2} \lambda^{-1} \sum_{i=1}^N |v_i|^2
\end{aligned}
\end{equation*}
We can bound the first term on the last line using 
Lemma \ref{lemma:s3-IP-1}. Hence,
\begin{equation*}
\left| r_N (t,Z_N)\right| \leq
\frac{1}{2} \lambda I_N \left( \psi_N^{-t} Z_N\right) +
\frac{1}{2} \lambda^{-1} \sum_{i=1}^N |v_i|^2
\end{equation*}
Replace $Z_N$ by $\psi_N^t Z_N$ on both sides and use the conservation
of energy to conclude.
\end{proof}

Combining Proposition \ref{prop:s3-Illner} and Lemma \ref{lemma:s3-IP-2},
we obtain:
\begin{corollary}
\label{cor:s3-IP-3}
For $\textnormal{ a.e. } Z_N = (X_N,V_N) \in\mathcal{D}_N$ and all
$\lambda>0$, we have
\begin{equation}
\label{eq:s3-IP-3}
\sum_k \left| \omega_k \cdot \left( v_{j_k} (t_k^-) -
v_{i_k} (t_k^-) \right) \right| \leq 2 \lambda^{-1} \sum_{i=1}^N
\left( \lambda^2 |x_i|^2 + |v_i|^2 \right)
\end{equation}
where the sum $\sum_k$ is over all collisions along the trajectory
$\left\{ \psi_N^t Z_N \right\}_{t\in\mathbb{R}}$.
\end{corollary}

\begin{remark}
It is interesting to compare Corollary \ref{cor:s3-IP-3} against the
case where we count \emph{all} collisions equally, without the
weighting factor $\left| \omega \cdot \left(
v_j - v_i \right) \right|$. In that case the best one can do with current
technology is bound the number of collisions by super-exponential
functions of the number of particles,
e.g. growing faster than $N^N$; we refer to
\cite{BFK1998} for estimates of this type. Indeed, the authors of
\cite{BFK1998} rightly note that much better estimates can be proven
if not all collisions are counted. Note that unweighted collision estimates
are of direct interest for dynamical systems theory,
whereas kinetic theory is
more interested in finding good \emph{function spaces} whose 
associated norms may well contain
weights.
\end{remark}

\begin{remark}
Clearly, Corollary \ref{cor:s3-IP-3} represents the ``worst case''
behavior for a system of hard spheres. If the initial conditions
$Z_N$ are chosen ``randomly'' (with suitable scalings of $x,v$;
cf. low-density limit, \cite{CIP1994})
 then the left-hand side should typically
be \emph{much smaller} than the right-hand side, at least when collisions
are counted on finite time intervals. This statement
can be formalized and proved (in an average sense for suitable
$f_0$) on a small time
interval in the Boltzmann-Grad scaling, using the bounds from
the proof of Lanford's theorem. An interesting open question, which
should be addressed,
 is whether improvements can be obtained (on average)
for Corollary \ref{cor:s3-IP-3}, locally in time and
 away from local equilibria,
while assuming less than what is
required to prove Lanford's theorem.  Such a result by itself cannot
be expected to allow improvement of the time of convergence
 in Lanford's theorem, but may provide
relevant insights in that direction.
\end{remark}

\section{An averaging trick}
\label{sec:4}

The previous section was primarily concerned with weighted sums over
collisions which occur along a single trajectory
$\left\{ \psi_N^t Z_N \right\}_t$. However, as has been explained
in Section \ref{sec:2}, we are really interested in ensemble averages
over many trajectories. This is due to the physical fact that we cannot
say with any precision what the initial state $Z_N$ ``really'' is.
We will ``prove'' a spacetime estimate by averaging both sides
of (\ref{eq:s3-IP-3}) with respect to the \emph{same} measure
$f_N (0,Z_N) dZ_N$ and applying a change of variables on the left-hand
side. The change of variables as presented here is not entirely rigorous,
though we are confident that this approach can be converted into a rigorous
proof. An alternative, completely rigorous, proof of the virial-type estimate
(Proposition \ref{prop:s4-interaction})
has already been given. \cite{De2016T} 

\begin{remark}
We emphasize that the results
of this section are not new, nor are they especially novel except perhaps
in the style of presentation; indeed,
 estimates of the type shown here go back many
decades. In particular, in this section
 we will prove a virial-type
estimate for the second marginal which holds under finiteness of
second moments, but this estimate is \emph{not} uniform in the Boltzmann-Grad
scaling. The novelty of our contribution is precisely the fact that
virial estimates \emph{can} sometimes provide nontrivial information in
the Boltzmann-Grad scaling (not easily accessible by other means), 
but that discussion is deferred 
to Section \ref{sec:L}.
\end{remark}

We will find it helpful to define an auxiliary function,
\begin{equation}
W_N^{(i,j)} (Z_N) = \left|
\left( x_j - x_i\right)\cdot \left( v_j - v_i\right)\right|
\end{equation}
Observe that $W_N^{(i,j)} (Z_N) = \left|
\omega \cdot \left(v_j - v_i\right) \right| $ if
$Z_N \in \partial \mathcal{D}_N$ represents a collision between 
particles $i$ and $j$ with $x_j = x_i + \omega$.
For any $Z_N \in \mathcal{D}_N$ let $\tilde{i}_N (Z_N)$,
$\tilde{j}_N (Z_N)$ be chosen such that $0 < \left|x_{\tilde{i}_N}
-x_{\tilde{j}_N}\right|\leq \left|x_i-x_j\right|$ for all
$i\neq j$. This uniquely defines $\tilde{i}_N$,$\tilde{j}_N$ (up to
switching the two indices) for
$\textnormal{ a.e. }Z_N \in \mathcal{D}_N$ and also for
$\textnormal{ a.e. }Z_N \in \partial \mathcal{D}_N$. Let us finally
define
\begin{equation}
W_N (Z_N) = W_N^{\left(\tilde{i}_N (Z_N),\tilde{j}_N (Z_N)\right)}(Z_N)
\end{equation}
so that $W_N (Z_N)$ is always equal to the correct collision parameter
$\left| \omega \cdot (v_j-v_i)\right|$ (when binary collisions
are well-defined) globally along 
$\partial \mathcal{D}_N$.

The ``proof'' of the spacetime estimate is based on the following
observation: the collision sum $\sum_k$ on the left-hand side of 
(\ref{eq:s3-IP-3}) may be re-cast as an integral in time:
\begin{equation*}
\sum_k W_N \left( \psi_N^{t_k} Z_N\right) \equiv
\int_{\mathbb{R}} \delta_{\psi_N^t Z_N \in \partial \mathcal{D}_N}
W_N \left( \psi_N^t Z_N \right) dt
\end{equation*}
Average both sides with respect to $f_N (0,Z_N) dZ_N$.
\begin{equation*}
\begin{aligned}
& \int_{\mathcal{D}_N}
\left\{ \sum_k W_N \left( \psi_N^{t_k} Z_N\right) \right\} 
f_N (0,Z_N) dZ_N = \\
& \qquad \qquad = \int_{\mathcal{D}_N} 
\int_{\mathbb{R}} \delta_{\psi_N^t Z_N \in \partial \mathcal{D}_N}
W_N \left( \psi_N^t Z_N \right) f_N (0,Z_N) dt dZ_N
\end{aligned}
\end{equation*}
The double integral on the right-hand side reduces (by Fubini) to
an integral of ``something'' over $\partial \mathcal{D}_N$, due to
the delta-function and the identity
$f_N (t,Z_N) = f_N (0,\psi_N^{-t}Z_N)$. Unfortunately, making the
change of variables precise requires a technical application of the
divergence theorem and careful manipulation of delta functions.
(The proof of \cite{De2016T} avoids any mention of delta functions.)
Here we record the result of correct manipulations:
\begin{equation}
\label{eq:s4-identity}
\begin{aligned}
& \int_{\mathcal{D}_N}
\left\{ \sum_k W_N \left( \psi_N^{t_k} Z_N\right) \right\} 
f_N (0,Z_N) dZ_N = \\
& \qquad \qquad = 
\int_{\mathbb{R}} \int_{\partial \mathcal{D}_N} 
\left[W_N \left(Z_N \right)\right]^2 f_N (t,Z_N) d\sigma_N dV_N dt
\end{aligned}
\end{equation}
where $d\sigma_N dV_N$ represents the surface measure on
$\partial \mathcal{D}_N$.

To conclude, we bound the left hand side of (\ref{eq:s4-identity})
using Corollary \ref{cor:s3-IP-3}, then reduce both sides using
the symmetry of $f_N (t)$ and the definition of the marginals
of $f_N (t)$. We have also simplified the estimate by optimal
choice of the parameter $\lambda > 0$.

\begin{proposition}
\label{prop:s4-interaction}
For each $N\in\mathbb{N}$, let $f_N (0)$ be an initial probability
density on $\mathcal{D}_N$, which we assume to be symmetric under particle
interchange, and let $f_N (t,Z_N) = f_N (0,\psi_N^{-t} Z_N)$. Let
$f_N^{(s)} (t)$,  $1\leq s \leq N$, denote the $s$-marginal of $f_N (t)$.
Further assume that $f_N (0)$ is smooth and compactly supported
in the interior of $\mathcal{D}_N$. Then 
for all $2 \leq s \leq N$  there holds
\begin{equation}
\label{eq:s4-interaction-1}
\begin{aligned}
 & \sum_{1\leq i < j \leq s} \int_{-\infty}^\infty
\int_{\mathbb{R}^{ds} \times \mathbb{R}^{d(s-1)} \times \mathbb{S}^{d-1}}
\left| \omega \cdot \left( v_j-v_i \right) \right|^2 \times \\
& \qquad \times f_N^{(s)} (t,\dots,x_i,v_i,\dots,
x_i+ \omega,v_j,\dots)
  d\omega dX_s^{(j)} dV_s dt \leq \\
&  \leq  C_d
\frac{ s (s-1)}{N}
\left( \int |x|^2 f_N^{(1)}(0,x,v)dxdv\right)^{\frac{1}{2}}
\left( \int |v|^2 f_N^{(1)}(0,x,v)dxdv\right)^{\frac{1}{2}}
\end{aligned}
\end{equation}
where $dX_s^{(j)} = dx_1 \dots dx_{j-1} dx_{j+1} \dots dx_s$ and 
$C_d$ is a constant depending only on the dimension $d$.
\end{proposition}

\begin{remark}
It is possible to compare
Proposition \ref{prop:s4-interaction} with
Lemma 1 of \cite{IP1989}; several other related formulas may also be
found in \cite{I1989}.
\end{remark}

\begin{remark}
Note that, up to constants,
 Proposition \ref{prop:s4-interaction} is formally equivalent
to Corollary \ref{cor:s3-IP-3} because we can always fix a point
$Z_N^0 \in \mathcal{D}_N$ (such that the trajectory
$\left\{ \psi_N^t Z_N^0 \right\}_{t}$ is globally defined)
 and let
\begin{equation}
f_N (0,Z_N) = \frac{1}{N!} \sum_{\sigma \in \mathcal{S}_N} 
\delta_{Z_N = \sigma Z_N^0}
\end{equation}
Therefore, just as Corollary \ref{cor:s3-IP-3} is suboptimal for
``many'' choices of initial condition $Z_N$ (under suitable
scalings),
Proposition \ref{prop:s4-interaction} is suboptimal for most
physically relevant densities $f_N (0)$. Again this is not a
theorem but an empirical observation, rooted in the apparent success
of kinetic theory in describing physical reality.
\end{remark}

\section{Factorized Data}
\label{sec:fact}

In this section, and for the remainder of the work, we will
assume that spheres are re-scaled to have diameter
$\varepsilon > 0$ instead of diameter $1$. Thus the proper condition
defining $\mathcal{D}_N$ is
$|x_i - x_j| > \varepsilon$ instead of
$|x_i - x_j| > 1$.
\begin{equation}
\mathcal{D}_N = 
\left\{ Z_N = (X_N,V_N)\in \mathbb{R}^{dN} \times \mathbb{R}^{dN}
\left| \forall 1\leq i < j \leq N,\; |x_i - x_j|
> \varepsilon \right. \right\}
\end{equation}
 Moreover, we assume
the Boltzmann-Grad scaling $N \varepsilon^{d-1} = 1$; 
physically, this means that the mean free path for a typical particle
is of order one. Note that the total volume occupied by all
$N$ particles is of order $\varepsilon$; hence, the Boltzmann-Grad
limit describes a state of low density.

Suppose $f_0 (x,v) $ is a measurable function on
$\mathbb{R}^d \times \mathbb{R}^d$ with 
\begin{equation}
0 \leq f_0 (x,v) \in 
\left( L^1_{x,v} \cap 
L^\infty_x L^1_v \right) \left( \mathbb{R}^d \times
\mathbb{R}^d \right) 
\end{equation}
Furthermore let us suppose, for convenience, that
$f_0$ is a normalized probability distribution:
\begin{equation}
\int_{\mathbb{R}^{2d}} f_0 (x,v) dx dv = 1
\end{equation}
Then it is natural to define ``factorized'' states on
$\mathcal{D}_N$ in the following way:
\begin{equation}
f_N (0,Z_N) =
\mathcal{Z}_N^{-1} f_0^{\otimes N} (Z_N)
\mathbf{1}_{Z_N \in \mathcal{D}_N}
\end{equation}
Here $\mathcal{Z}_N$ is the partition function,
\begin{equation}
\mathcal{Z}_N = \int_{\mathbb{R}^{2dN}}
f_0^{\otimes N} (Z_N) \mathbf{1}_{Z_N \in \mathcal{D}_N}
dZ_N
\end{equation}
We will use the imprecise shorthand
$f_N (0) \sim f_0^{\otimes N}$ for such ``factorized''
initial data.
Then it is possible to prove the following \emph{pointwise}
estimate for the first marginal at $t=0$, valid for almost
every $(x,v) \in \mathbb{R}^{2d}$ under the Boltzmann-Grad
scaling $N \varepsilon^{d-1} = 1$, for all small
enough $\varepsilon > 0$ depending only on $f_0$ and $d$:
\begin{equation}
\left| f_N^{(1)} (0,x,v) - f_0 (x,v) \right| \leq
C_d \left\Vert f_0 \right\Vert_{L^\infty_x L^1_v} 
f_0 (x,v) \varepsilon
\end{equation}
Similar pointwise convergence estimates
(also of order $\varepsilon$) are available for higher order
marginals $f_N^{(s)} (0)$ as well, as long as $s$ is fixed
as $N\rightarrow \infty$. We refer to
\cite{GSRT2014} or \cite{De2016T} for detailed proofs.

The $\mathcal{O}(\varepsilon)$ convergence rate for the first marginal
at $t=0$ arises from careful partition function estimates, and 
is intuitively due to the fact that the total volume occupied by
all particles is roughly $N \varepsilon^d$, that is,
$\mathcal{O}(\varepsilon)$ in the Boltzmann-Grad scaling
$N\varepsilon^{d-1} = 1$. It is not hard to derive 
\emph{worse} convergence rates under weaker regularity assumptions,
e.g. $f_0 \in L^p_x L^1_v$ with $d<p<\infty$.
However, as far as we are
aware, one does not obtain an error which is
\emph{smaller} than $\mathcal{O}(\varepsilon)$ even if
$f_0$ is in the Schwartz class, or jointly Gaussian in $x$ and $v$. 

Let us now assume that $f_0$ is in the Schwartz class and let
$f(t)$ be the solution of Boltzmann's equation (on a small time
interval) with hard sphere interaction and $f(0) = f_0$:
\begin{equation}
\left( \partial_t + v \cdot \nabla_x \right)
f (t,x,v) = Q^+ (f,f) (t,x,v) - Q^- (f,f) (t,x,v)
\end{equation}
\begin{equation}
Q^+ (f,f) 
= \int_{\mathbb{R}^d \times \mathbb{S}^{d-1}}
d\omega dv_2 \left| \omega \cdot (v - v_2) \right|
f (t,x,v^*) f(t,x,v_2^*)
\end{equation}
\begin{equation}
Q^- (f,f) 
= \int_{\mathbb{R}^d \times \mathbb{S}^{d-1}}
d\omega dv_2 \left| \omega \cdot (v - v_2) \right|
f (t,x,v) f(t,x,v_2)
\end{equation}
Here we define the collisional change of variables,
for a unit vector
 $\omega \in \mathbb{S}^{d-1}$ and any
$v,v_2 \in \mathbb{R}^d$,
\begin{equation}
\begin{aligned}
v^* & = v + \omega \omega \cdot ( v_2 - v) \\
v_2^* & = v_2 - \omega \omega \cdot (v_2 - v)
\end{aligned}
\end{equation}
In the mid 1970s, Oscar Lanford showed that if $f_0 (x,v)$ is nice
enough (say smooth with compact support on $\mathbb{R}^{2d}$) then
the function $f_N (t,Z_N) = f_N (0,\psi_N^{-t} Z_N)$ with
initial data $f_N (0) \sim f_0^{\otimes N}$ 
has first marginal converging to the solution of
Boltzmann's equation on a small time interval in the
Boltzmann-Grad scaling $N\varepsilon^{d-1} = 1$.
\cite{L1975}
Namely, for some small $T>0$ depending on $f_0$,
any $t \in [0,T]$,
and almost every $x,v \in \mathbb{R}^d$,
\begin{equation}
\lim_{N \rightarrow \infty}
f_N^{(1)} (t,x,v) =
f(t,x,v)
\end{equation}
Moreover it follows from Lanford's proof that the higher order
marginals, say $f_N^{(s)} (t)$, converge to the tensor
products $f(t)^{\otimes s}$. Lanford's theorem was the first rigorous
justification of kinetic theory from deterministic
Newtonian mechanics, and his proof remains the basis of most of the
more recent developments in first principles derivations of
collisional kinetic equations.

\begin{remark}
We must point out that a completely different point
of view, pioneered by Kac and McKean, is to start with a
\emph{stochastic} model in which the microscopic
 position coordinates are
essentially hidden variables. In these models, the impact parameter
is automatically random even in the $N$ particle system, and
much more detailed results are available. Most notably, unlike the
Lanford theorem, the convergence to Boltzmann
 is often proven globally in
time. We refer to
\cite{MM2013} and references therein for more details.
\end{remark}

Within the past few years, a number of authors 
have worked to make Lanford's theorem into a more
quantitative result. Arguably the most notable contribution along
these lines is \cite{GSRT2014}, in which convergence rates
of order $\mathcal{O} \left(\varepsilon^{\frac{d-1}{d+1}-} \right)$ were
obtained. (The convergence can be proven in $L^\infty_{x,v}$
for the first
marginal, but not for higher marginals, due to issues of
irreversibility which we do not discuss here.)
It is not established in \cite{GSRT2014} whether the convergence rate
obtained therein is optimal (nor does there seem to be a particularly
compelling reason that it should be optimal). Intuitively it should
not be possible to obtain an error that is much smaller than
$\mathcal{O}(\varepsilon)$, since such small errors cannot even be
proven at $t=0$. We are not aware of any proof in the literature that the
convergence rate in Lanford's theorem can be
$\mathcal{O}\left(\varepsilon^{1-}\right)$; however, what we will
prove here is that the $L^\infty_{x,v}$ error (of the first marginal)
certainly cannot be
much \emph{smaller} than $\mathcal{O}\left(\varepsilon\right)$.
The proof is based on a virial-type inequality
similar to (\ref{eq:s4-interaction-1}).

\section{On Convergence Rates in Lanford's Theorem}
\label{sec:L}

In this section we will establish a lower bound on the convergence
rate in Lanford's theorem; throughout our discussion, 
we will actually assume Maxwellian tails
jointly in $x$ and $v$, as in \cite{IP1986,IP1989}. 
Before stating our main result, a few comments are in order.
The first is that conservation laws may automatically imply bounds
from below: after all, we certainly have
\begin{equation}
\int_{\mathbb{R}^{2d}} |v|^2 f_N^{(1)} (t,x,v) dx dv
= \int_{\mathbb{R}^{2d}} |v|^2 f_N^{(1)} (0,x,v) dx dv
\end{equation}
Therefore, in the event that $f_N (0)$ is chosen to satisfy
the following inequality:
\begin{equation}
\liminf_{N\rightarrow \infty}
\varepsilon^{-1} 
\left| \int_{\mathbb{R}^{2d}}
|v|^2 f_N^{(1)} (0,x,v) dx dv
- \int_{\mathbb{R}^{2d}} |v|^2 f_0 (x,v) dx dv
\right| > 0
\end{equation}
then $f_N^{(1)}(t)$ cannot be within
$o \left(\varepsilon \right)$ of $f(t)$, since  the kinetic
energy witnesses a slower convergence rate. 
On the other hand, by the example below, 
 it is not hard to construct factorized densities
$f_N (0)$ which satisfy both
\begin{equation}
\label{eq:L-aaaa}
\int_{\mathbb{R}^{2d}} |v|^2 f_N^{(1)} (0,x,v) dx dv
= \int_{\mathbb{R}^{2d}} |v|^2 f_0 (x,v) dx dv
\end{equation}
and
\begin{equation}
\label{eq:L-bbbb}
\int_{\mathbb{R}^{2d}} v f_N^{(1)} (0,x,v) dx dv
= \int_{\mathbb{R}^{2d}} v f_0 (x,v) dx dv
\end{equation}
Therefore the conserved moments do not \emph{always} place a lower
bound on the convergence rate in Lanford's theorem.

\begin{example}
\label{ex:ex1}
Fix a large number $K$.
Let $\mathcal{J}$ be a (possibly uncountable) index set and let $\mu$ be
a probability 
measure on on $\mathcal{J}$. For $\mu\textnormal{-a.e.}$
$\alpha \in \mathcal{J}$, measurably with respect to $\alpha$,
 we pick a non-negative measurable function
$\rho_\alpha (x)$ with $\int \rho_\alpha (x) dx = 1$
and $\left\Vert \rho_\alpha \right\Vert_\infty \leq K$, and a
non-negative measurable function $m_\alpha (v)$ with
$\int m_\alpha (v) dv = 1$, $\int v m_\alpha (v) dv = 0$ and
$\int |v|^2 m_\alpha (v) dv = 1$. Then if we write
\begin{equation}
f_0 (x,v) = \int_{\mathcal{J}} 
d\mu (\alpha) \rho_\alpha (x) m_\alpha (v)
\end{equation}
and let $f_N (0) \sim f_0^{\otimes N}$ as in Section \ref{sec:fact},
then we have both (\ref{eq:L-aaaa}) and (\ref{eq:L-bbbb}). 
\end{example}

The second observation is that (under natural decay assumptions)
 we can place lower bounds on
the $L^\infty$ error, up to powers of $\log \frac{1}{\varepsilon}$,
 by placing lower bounds on a weighted
$L^1$ error. To see why, observe that if both
$f_N^{(1)} (t)$ and $f(t)$ have Maxwellian tails jointly 
in $x$ and $v$ (as in the work of \cite{IP1986,IP1989}) then
there exists a number $k=k(t)>0$ (depending on $f$)
such that the following estimate holds
for all $C>0$:
\begin{equation}
\begin{aligned}
& \int_{\mathbb{R}^{2d}} \left( |x|^2 + |v|^2 \right)
\left| f_N^{(1)} (t) - f(t) \right| dx dv
\lesssim \\
& \qquad \qquad \qquad
\lesssim \left( C \log \frac{1}{\varepsilon} \right)^{d+1}
\left\| f_N^{(1)} (t,x,v) - f (t,x,v) \right\|_{L^\infty_{x,v}}
+ \varepsilon^{kC}
\end{aligned}
\end{equation}
Hence we will not actually mention the $L^\infty_{x,v}$
norm in stating our main result.

Finally, we remark that for any solution $f(t)$ of Boltzmann's
equation having enough regularity and decay in $x$ and $v$,
there holds
\begin{equation}
\label{eq:L-Bvir}
\frac{d}{dt}
\int_{\mathbb{R}^{2d}} \left( x\cdot v - |v|^2 t\right)
f (t,x,v) dx dv  = 0
\end{equation}
We will not actually use any regularity properties of
Boltzmann's equation aside from (\ref{eq:L-Bvir}). Therefore,
instead of trying to find optimal conditions which
guarantee (\ref{eq:L-Bvir}), we will simply include
(\ref{eq:L-Bvir}) as a hypothesis in the theorem.

\begin{theorem}
\label{thm:MainThm}
Let $f_N (t,Z_N)$ be a non-negative solution of the Liouville equation for
$N$ identical hard spheres of diameter $\varepsilon > 0$,
for each $N$ in the Boltzmann-Grad scaling $N \varepsilon^{d-1} = 1$.
We suppose that $f_N$ is symmetric under particle interchange,
 that $f_N$ is a probability density, and that the following
moment estimate holds:
\begin{equation}
\sup_N \frac{1}{N} \int_{\mathcal{D}_N}
\sum_{i=1}^N \left( |x_i|^2 + |v_i|^2 \right)
f_N (0,Z_N) dZ_N < \infty
\end{equation}
Additionally, assume that $f(t)$ is a solution of the Boltzmann
equation which satisfies
\begin{equation}
\label{eq:L-a}
\forall t\in [0,T],\quad
\int_{\mathbb{R}^{2d}}
\left( x\cdot v - |v|^2 t\right)
f (t,x,v) dx dv =
\int_{\mathbb{R}^{2d}} x\cdot v f(0,x,v) dx dv
\end{equation}
Assume, moreover, that $f(t)$ does not satisfy the free transport
equation on any open subinterval of $[0,T]$. Finally, assume that
\begin{equation}
\label{eq:L-cc}
\lim_{N\rightarrow \infty} f_N^{(1)} (t) = f(t)
\end{equation}
holds in the sense of distributions for each $t \in [0,T]$.
 Then for any $T_1,T_2$ with
$0 \leq T_1 < T_2 \leq T$ we have
\begin{equation}
\label{eq:L-conc}
\liminf_{N\rightarrow \infty} \varepsilon^{-1}
\sup_{t \in [T_1,T_2]}
\int_{\mathbb{R}^{2d}} \left( |x|^2 + |v|^2 \right)
\left| f_N^{(1)} (t) - f(t) \right| dx dv > 0
\end{equation} 
\end{theorem}

\begin{remark}
The requirement that $f$ is not a solution of free transport is a
technical condition; it excludes local Maxwellian functions,
e.g. $f (t) = c e^{-|x - v t|^2} e^{-|v|^2}$. It is expected that
optimal $\mathcal{O} (\varepsilon)$ convergence rates hold for 
such solutions, but this is not a part of the theorem.
\end{remark}

\begin{remark}
The conditions of Theorem \ref{thm:MainThm} are satisfied if
$f_0 : \mathbb{R}^d \times \mathbb{R}^d \rightarrow \mathbb{R}$
is a smooth, compactly supported,
 non-negative probability density function,
 with associated solution $f(t)$ of Boltzmann's
equation; and,
$f_N (0) \sim f_0^{\otimes N}$ as in Section \ref{sec:fact},
and $T$ is the small time appearing in the original theorem
of Lanford. \cite{L1975}
\end{remark}

\begin{remark}
The supremum over $[T_1,T_2]$ in (\ref{eq:L-conc}) can actually be
replaced by a supremum over the two-point set
$\left\{ T_1, T_2 \right\}$ (the proof is the same).
\end{remark}

\begin{proof}
We will assume that (\ref{eq:L-conc}) fails for some
$T_1,T_2$ in order to reach a contradiction.
We have the following virial identity
(see the proof of Proposition 1.3.5 in \cite{De2016T},
or average out (\ref{eq:s3-Illner}) in the manner
of Section \ref{sec:4}):
\begin{equation}
\begin{aligned}
& \int_{\mathbb{R}^{2d}}
\left( x \cdot v - |v|^2 T_2 \right)
f_N^{(1)} (T_2) dx dv - \int_{\mathbb{R}^{2d}}
\left( x\cdot v - |v|^2 T_1 \right) f_N^{(1)} (T_1) dx dv \\
& \qquad \qquad =
C_d \frac{N-1}{2} \varepsilon^d \int_{T_1}^{T_2} 
\int_{\mathbb{R}^d \times \mathbb{R}^d \times \mathbb{R}^d
\times \mathbb{S}^{d-1}}
\left| \omega \cdot (v_2 - v_1) \right|^2 \times \\
&\qquad \qquad \qquad \qquad \qquad \qquad
\times f_N^{(2)} (t,x_1,v_1,x_1 + \varepsilon \omega,v_2)
d\omega dx_1 dv_1 dv_2 dt
\end{aligned}
\end{equation}
Therefore, using the Boltzmann-Grad scaling
$N\varepsilon^{d-1} = 1$ and 
(\ref{eq:L-a}),  we obtain:
\begin{equation}
\begin{aligned}
& \int_{T_1}^{T_2} \int_{\mathbb{R}^d \times \mathbb{R}^d
\times \mathbb{R}^d \times \mathbb{S}^{d-1}}
\left| \omega \cdot (v_2 - v_1) \right|^2 
f_N^{(2)} (t,x_1,v_1,x_1 + \varepsilon \omega,v_2)
d\omega dx_1 dv_1 dv_2 dt \\
& \quad \qquad  \leq 
C_{d,T} 
\varepsilon^{-1} 
\sup_{t \in [T_1,T_2]}
\int_{\mathbb{R}^{2d}}
\left( |x|^2 + |v|^2 \right)
\left| f_N^{(1)} (t,x,v) - f (t,x,v) \right| dx dv
\end{aligned}
\end{equation}
Here we have used the triangle inequality and the
fact that (\ref{eq:L-a}) holds with $t=T_1$ and $t=T_2$.

Hence if (\ref{eq:L-conc}) is not true then there exists a subsequence
$N^\prime$ (depending on $T_1,T_2$) such that
\begin{equation}
\label{eq:L-b}
\lim_{N^\prime}
\int_{T_1}^{T_2}
\int \left| \omega \cdot (v_2 - v_1) \right|^2
f_{N^\prime}^{(2)} (t,x_1,v_1,x_1+\varepsilon \omega,v_2)
d\omega dx_1 dv_1 dv_2 dt = 0
\end{equation}
The remainder of the proof consists in showing that (\ref{eq:L-b})
combined with (\ref{eq:L-cc})
implies that $f(t)$ satisfies the free transport equation
on $[T_1,T_2]$.

Define the transport semigroup, which acts on a function
$g(x,v)$, by the formula
\begin{equation}
\left[ \mathcal{T} (\tau) g \right]
(x,v) = g (x-\tau v,v)
\end{equation}
Also define the (lowest order) BBGKY collision operator
\begin{equation}
C_2 f_N^{(2)} (t,x_1,v_1) =
\int_{\mathbb{R}^d \times \mathbb{S}^{d-1}}
\omega \cdot (v_2 - v_1) 
f_N^{(2)} (t,x_1,v_1,x_1 + \varepsilon \omega,v_2)
d\omega dv_2
\end{equation}
Then we have (see \cite{GSRT2014}, note that there is
a factor of $(N-1)\varepsilon^{d-1}$ which we ignore in view
of the Boltzmann-Grad scaling):
\begin{equation}
\left( \partial_t + v \cdot \nabla_x
\right) f_N^{(1)} (t)
= C_2 f_N^{(2)} (t)
\end{equation}
Therefore for $t \in [T_1,T_2]$ we have the Duhamel representation
\begin{equation}
f_N^{(1)} (t) =
\mathcal{T} (t-T_1) f_N^{(1)} (T_1)
+ \int_{T_1}^t \mathcal{T} (t-\tau)
C_2 f_N^{(2)} (\tau) d\tau
\end{equation}

Let us employ the usual $L^2$ inner product
\begin{equation}
\left< \varphi, f \right> =
\int_{\mathbb{R}^{2d}} \varphi (x,v) f (x,v) dx dv
\end{equation}
Then for any smooth compactly supported function $\varphi$
and any $t \in [T_1,T_2]$ we have
\begin{equation}
\begin{aligned}
& \left< \varphi,
f_N^{(1)} (t) - \mathcal{T} (t-T_1) f_N^{(1)} (T_1) \right>
 = \int_{T_1}^t \left<
\varphi, \mathcal{T} (t-\tau) C_2 f_N^{(2)} (\tau) \right> d\tau
\end{aligned}
\end{equation}
Now by (\ref{eq:L-cc}) and the assumption that $f (t)$ does 
not solve the free transport equation on $[T_1,T_2]$, in order
to reach a contradiction it suffices to show:
\begin{equation}
\lim_{N^\prime}
\int_{T_1}^t \left<
\varphi, \mathcal{T} (t-\tau) C_2 f_N^{(2)} (\tau) \right> d\tau
= 0
\end{equation}

Using duality we have
\begin{equation}
\begin{aligned}
\int_{T_1}^t \left< \varphi,
\mathcal{T} (t-\tau) C_2 f_N^{(2)} (\tau) \right> d\tau
& = \int_{T_1}^t \left< \mathcal{T} (-(t-\tau)) \varphi,
C_2 f_N^{(2)} (\tau) \right> d\tau \\
& = 
\int_{T_1}^t \left<
\Phi (\tau),C_2 f_N^{(2)} (\tau) \right> d\tau
\end{aligned}
\end{equation}
Here we have defined (considering $t$ fixed)
\begin{equation}
\Phi (\tau) = 
\mathcal{T} (-(t-\tau)) \varphi
\end{equation}
Using the definition of the BBGKY collision operator, we have
\begin{equation}
\label{eq:L-rrr}
\begin{aligned}
& \int_{T_1}^t \left<
\Phi (\tau),C_2 f_N^{(2)} (\tau) \right> d\tau = \\
& \int_{T_1}^t
\int d\tau d\omega dx_1 dv_1 dv_2 
\; \omega \cdot (v_2 - v_1)
\Phi (\tau,x_1,v_1)
f_N^{(2)} (\tau,x_1,v_1,x_1+\varepsilon \omega,v_2)
\end{aligned}
\end{equation}

We recall the collisional change of variables
\begin{equation}
\begin{aligned}
v_1^* & = v_1 + \omega \omega \cdot (v_2-v_1) \\
v_2^* & = v_2 - \omega \omega \cdot (v_2-v_1)
\end{aligned}
\end{equation}
as well as the boundary condition
\begin{equation}
f_N^{(2)} (\tau,x_1,v_1^*,x_1+\varepsilon \omega,
v_2^*) =
f_N^{(2)} (\tau,x_1,v_1,x_1+\varepsilon \omega,v_2)
\end{equation}
hence what we need to control is the \emph{difference}
\begin{equation}
\Phi (\tau,x_1,v_1) - \Phi (\tau,x_1,v_1^*)
\end{equation}
Of course for $T_1 \leq \tau \leq t$ we have
\begin{equation}
\begin{aligned}
\left|\Phi (\tau,x_1,v_1) - \Phi (\tau,x_1,v_1^*) \right|
& \leq \left\Vert \nabla_v \Phi (\tau) \right\Vert_{L^\infty_{x,v}}
\left|v_1 - v_1^* \right| \\
& =  \left\Vert \nabla_v \Phi (\tau) \right\Vert_{L^\infty_{x,v}}
\left| \omega \cdot (v_2 - v_1) \right| \\
& \leq C_T \left\Vert
\nabla_{x,v} \varphi \right\Vert_\infty
\left| \omega \cdot (v_2 - v_1) \right|
\end{aligned}
\end{equation}
(This trick is inspired by related work due to
Cercignani using the Bony functional. \cite{Ce2005,Ce1995,Ce1992})

Hence the right hand side of (\ref{eq:L-rrr}) is controlled
by the following quantity:
\begin{equation}
C_T
 \left\Vert \nabla_{x,v} \varphi \right\Vert_\infty
\int_{T_1}^t \int d\tau d\omega
dx_1 dv_1 dv_2 
\left| \omega \cdot (v_2 -v_1) \right|^2
f_N^{(2)} (\tau,x_1,v_1,x_1 + \varepsilon \omega,v_2)
\end{equation}
which is going to zero (along the subsequence
$N^\prime$) due to (\ref{eq:L-b}).
\end{proof}

\section{Acknowledgements}
\label{sec:acknowledgements}

This paper is largely based on work completed for the author's
dissertation at New York University. The partial manuscript was
completed under a postdoctoral fellowship at the University
of Texas at Austin, for which I am most appreciative.
 I would like to thank my
PhD advisor, Nader Masmoudi, as well as Pierre Germain,
for their advice and comments. I would also like to thank
Nata{\v s}a Pavlovi{\' c} for reading an early version of this
manuscript and providing insightful feedback.
Additionally, I wish to indicate my appreciation to the anonymous
referee(s) for helpful comments following careful reading of the present version.
Finally I would like to thank the organizers of this
special session of JMM 2017 for the invitation, without which
this manuscript would most likely never have reached
its current state of completion.

\bibliography{Virial}

\end{document}